\documentclass[11pt,a4paper]{article}

\usepackage[T1]{fontenc}
\usepackage{lmodern}
\usepackage{microtype}
\usepackage{mathrsfs,mathtools,amsthm,amssymb,bm}
\usepackage{aliascnt}
\allowdisplaybreaks[2]
\usepackage{graphicx,xcolor,tikz}
\usetikzlibrary{calc,positioning,fit}
\usepackage[left=25mm,right=25mm,top=25mm,bottom=25mm]{geometry}
\usepackage[hang,flushmargin]{footmisc}

\usepackage[inline]{enumitem}
\usepackage[square,numbers,sort&compress]{natbib}
\usepackage{algorithm}
\usepackage[noend]{algpseudocode}
\usepackage[hidelinks,
  pdftitle={Extensions of Erdos's 1962 theorem on non-Hamiltonian graphs},
  pdfauthor={Xu Liu, Bo Ning, and Tao Wang},
  pdfsubject={Extremal Hamiltonian graph theory and feasible graph parameters},
  pdfkeywords={Hamilton cycle, Hamilton path, Hamilton-connected graph, feasible parameter, spectral radius, Kelmans operation}]{hyperref}
\usepackage[capitalise,noabbrev]{cleveref}

\newtheorem{theorem}{Theorem}[section]
\newaliascnt{lemma}{theorem}
\newtheorem{lemma}[lemma]{Lemma}
\aliascntresetthe{lemma}
\newaliascnt{corollary}{theorem}
\newtheorem{corollary}[corollary]{Corollary}
\aliascntresetthe{corollary}
\newtheorem{claim}{Claim}

\theoremstyle{definition}
\newtheorem{definition}{Definition}
\newtheorem{remark}{Remark}
\newtheorem{problem}{Problem}

\crefname{claim}{Claim}{Claims}
\crefname{lemma}{Lemma}{Lemmas}
\crefname{corollary}{Corollary}{Corollaries}
\crefname{problem}{Problem}{Problems}
\crefname{conjecture}{Conjecture}{Conjectures}

\makeatletter
\renewcommand*{\theHALG@line}{\thealgorithm.\arabic{ALG@line}}
\makeatother

\DeclareMathOperator{\ex}{\mathsf{ex}}
\DeclareMathOperator{\EX}{\mathsf{EX}}
\DeclareMathOperator{\spex}{\mathsf{spex}}
\DeclareMathOperator{\SPEX}{\mathsf{SPEX}}
\newcommand{\HC}{\mathsf{HC}}
\newcommand{\cl}{\operatorname{cl}}

\begin{document}

\title{Extensions of Erd\H{o}s's 1962 theorem on non-Hamiltonian graphs}
\author{
Xu Liu\thanks{College of Computer Science, Nankai University, Tianjin 300350, P.~R. China.}
\and
Bo Ning\thanks{Corresponding author. College of Computer Science, Nankai University, Tianjin 300350, P.~R. China. \nolinkurl{Email:bo.ning@nankai.edu.cn}. Partially supported by the National Natural Science Foundation of China (No. 12371350).~\href{https://orcid.org/0000-0002-9622-5567}{ORCID: 0000-0002-9622-5567}.}
\and
Tao Wang\thanks{Center for Applied Mathematics, Henan University, Kaifeng 475004, P.~R. China. \nolinkurl{Email:wangtao@henu.edu.cn}.\href{https://orcid.org/0000-0001-9732-1617}{ORCID: 0000-0001-9732-1617}.}}
\date{}
\maketitle

\begin{abstract}
For a positive integer $k$, a graph property $\mathcal{H}$, and a graph parameter $\mathcal{P}$, let $\operatorname{ex}_{\mathcal{P}}(n, \mathcal{H}; \delta \geq k)$ denote the maximum value of $\mathcal{P}$ over all $n$-vertex graphs with minimum degree at least $k$ that do not possess the property $\mathcal{H}$. The corresponding extremal families are denoted by $\operatorname{EX}_{\mathcal{P}}(n, \mathcal{H}; \delta \geq k)$. For two disjoint graphs $H_1$ and $H_2$, let $H_1 \cup H_2$ denote their disjoint union, and let $H_1 \vee H_2$ denote their join.

In 1962, Erd\H{o}s established a classical theorem on the maximum number of edges in a non-Hamiltonian graph with prescribed order and minimum degree. Motivated by recent work on feasible graph parameters in \cite{ALNS2025}, we prove several extensions of Erd\H{o}s's 1962 theorem on non-Hamiltonian graphs: 
\begin{enumerate}[label = (\roman*)]
\item For $k \ge 1$, $n \geq 2k + 1$ and a feasible parameter $\mathcal{P}$, we have
$\ex_{\mathcal{P}}(n, C_n; \delta \geq k) = \max\{\mathcal{P}(K_s \vee (sK_1 \cup K_{n-2s})): k \leq s \leq \lfloor\tfrac{n-1}{2}\rfloor\}$,
and
$\EX_{\mathcal{P}}(n, C_n; \delta \geq k) \subseteq \{K_s \vee (sK_1 \cup K_{n-2s}): k \leq s \leq \lfloor\tfrac{n-1}{2}\rfloor\}$.

\item For $k \ge 1$, $n \geq 2k + 2$ and a feasible parameter $\mathcal{P}$, we have
$
\ex_{\mathcal{P}}(n, P_n; \delta \ge k) = \max\{\mathcal{P}(K_{s-1} \vee (sK_1 \cup K_{n-2s+1})): k+1 \leq s \leq \lfloor\tfrac{n}{2}\rfloor\},
$
and
$
\EX_{\mathcal{P}}(n, P_n; \delta \ge k) \subseteq \{K_{s-1} \vee (sK_1 \cup K_{n-2s+1}): k+1 \leq s \leq \lfloor\tfrac{n}{2}\rfloor\}.
$
\end{enumerate}
The first result gives a common generalization of the extremal theorem due to Erd\H{o}s and its spectral analogues. As direct applications, we obtain complete solutions to open problems raised in the literature since 2016, thereby improving nearly all related prior results in this direction. Our proof technique differs somewhat from those in \cite{LN2016,N2016}. 
We also prove an analogous theorem for the Hamiltonian-connected graphs and obtain a result which extends a theorem of F\"{u}redi, Kostochka, and Luo \cite{FKL2018} on Hamilton cycles. Compared with \cite{ALNS2025}, our approach overcomes the difficulty that the Kelmans operation may change the minimum degree of a graph.
\end{abstract}

%-----------------------------------------------------------------------
\section{Introduction}

A \emph{Hamilton cycle} in a graph is a spanning cycle, that is, a cycle that passes through every vertex exactly once. The Hamilton cycle problem asks whether a given graph contains a Hamilton cycle. The subject has a classical connection with the Four Color Problem: Tait observed that the Four Color Theorem would follow if every 3-connected cubic planar graph were Hamiltonian, an assertion now known as the Tait conjecture; see, for example, \cite[Chapter~10]{BM2008}. It is also computationally fundamental, since the directed and undirected Hamilton circuit problems appear in Karp's classical list of 21 NP-complete problems \cite{Karp1972}. Degree conditions are among the central tools in the study of Hamiltonicity. One of the classical cornerstones of this type is Dirac’s theorem \cite{D1952}, which states that every graph $G$ of order $n\ge3$ with $\delta(G)\ge n/2$ contains a Hamilton cycle.

In the extremal setting, Ore's theorem \cite{O1960} on Hamilton cycles implies that $\operatorname{ex}(n,C_n)=\binom{n-1}{2}+1$. By incorporating the minimum degree as a new parameter, Erd\H{o}s \cite{E1962} proved the following fundamental result in 1962.

\begin{theorem}[Erd\H{o}s \cite{E1962}]\label{thm:erdos1962}
Let $k$ and $n$ be positive integers with $1\le k\le \lfloor(n-1)/2\rfloor$, and let $G$ be a graph of order $n$ with $\delta(G)\ge k$. If
\[
e(G)>\max\left\{\binom{n-k}{2}+k^2,
\binom{n-\lfloor(n-1)/2\rfloor}{2}+\left\lfloor\frac{n-1}{2}\right\rfloor^2\right\},
\]
then $G$ contains a Hamilton cycle.
\end{theorem}

A simplified version of Erd\H{o}s's 1962 theorem appears as an exercise in West's textbook \cite[Exercise~7.2.28]{W2001}: if $G$ is a graph of order $n$ with $\delta(G)\ge k$, $n\ge6k$, and $e(G)>\binom{n-k}{2}+k^2$, then $G$ contains a Hamilton cycle. In 2016, Li and Ning \cite[Lemma~2]{LN2016} studied an extremal problem in spectral graph theory and proved a stability version of Erd\H{o}s's theorem. More precisely, they proved that, for $k\ge1$ and $n\ge6k+5$, if $G$ is a graph of order $n$ with $\delta(G)\ge k$ and $e(G)>\binom{n-k-1}{2}+(k+1)^2$, then $G$ is Hamiltonian unless $G$ is a spanning subgraph of $K_k\vee(kK_1\cup K_{n-2k})$ or $K_1\vee(K_k\cup K_{n-k-1})$. Here, $H_1\cup H_2$ denotes the disjoint union of vertex-disjoint graphs $H_1$ and $H_2$, and $H_1\vee H_2$ denotes their join. A similar stability result was obtained independently by F\"uredi, Kostochka, and Luo \cite{FKL2017}.

A clique version of Erd\H{o}s's 1962 theorem was obtained by F\"uredi, Kostochka, and Luo \cite[Theorem~6]{FKL2018}. We say that a graph is \emph{$\ell$-Hamiltonian} if every linear forest $F\subseteq G$ with $e(F)=\ell$ is contained in a Hamilton cycle. For this class of graphs, F\"uredi, Kostochka, and Luo \cite{FKL2019} proved an analogue of Erd\H{o}s's theorem, together with a stability result for the number of cliques in non-$\ell$-Hamiltonian graphs. Erd\H{o}s-type theorems have also been investigated for special graph families. In particular, Li, Ning, and Peng \cite[Theorem~1.3]{LNP2018} proved a sharp edge condition for Hamiltonicity in 2-connected claw-free graphs with prescribed order and minimum degree, and characterized the exceptional graph at the threshold. We refer to \cite{FKV2016,FKL2017,FKLV2018,FKL2018,LN2023,MY2024} for further stability results concerning Hamiltonian properties, paths, and cycles.

Before stating our main results, we introduce the concept of a \emph{feasible parameter}, whose original definition is only on connected graphs. Ai, Lei, Ning, and Shi \cite{ALNS2025} proved that the number of edges, the spectral radius, and the signless Laplacian spectral radius are feasible parameters for connected graphs.

For $v\in V(G)$, let $N_G(v)$ denote the neighborhood of $v$. The graph operation below goes back to Kelmans \cite{Kelmans1981}. Following the orientation and notation of \cite{ALNS2025}, we let $x,y$ be distinct vertices of $G$, $W=N_G(x)\setminus N_G[y]$, and define
\[
G[x\to y]=G-\{xw:w\in W\}+\{yw:w\in W\},
\tag{1.1}\label{eq:kelmans}
\]
where $N_G[y]=N_G(y)\cup\{y\}$. We call $G[x\to y]$ the graph obtained from $G$ by the \emph{Kelmans operation from $x$ to $y$}. Thus, every private neighbor of $x$ with respect to $y$ is transferred from $x$ to $y$ (after the operation).

\begin{definition}[Ai, Lei, Ning, and Shi \cite{ALNS2025}]\label{def:connected-feasible}
A graph parameter $\mathcal P$ is \emph{feasible on connected graphs} if the following conditions hold for every connected graph $G$:
\begin{enumerate}[label=(P\arabic*)]
\item $\mathcal P(G+xy)>\mathcal P(G)$ for every nonedge $xy$ of $G$;
\item $\mathcal P(G[x\to y])\ge\mathcal P(G)$ for every ordered pair of distinct vertices $x,y\in V(G)$.
\end{enumerate}
\end{definition}

One first contribution of this paper is to extend the previous definition to arbitrary graphs, including disconnected graphs. If $F\subseteq\binom{V(G)}2\setminus E(G)$, then $G+F$ denotes the graph with vertex set $V(G)$ and edge set $E(G)\cup F$.

\begin{definition}\label{def:feasible}
A graph parameter $\mathcal P$ is \emph{feasible} if the following conditions hold for every graph $G$.
\begin{enumerate}[label=(P\arabic*)]
\item $\mathcal P(G+F)>\mathcal P(G)$ for every nonempty set
$F\subseteq\binom{V(G)}2\setminus E(G)$ such that $G+F$ is connected;
\item $\mathcal P(G[x\to y])\ge\mathcal P(G)$ for every ordered pair of distinct vertices $x,y\in V(G)$.
\end{enumerate}
\end{definition}

We shall use two graph operations: addition of non-edges and the Kelmans operation in \eqref{eq:kelmans}. The following two results of Ai, Lei, Ning, and Shi illustrate the feasible-parameter method. The first substantially generalizes Kopylov's theorem \cite{Kopylov1977}, while a consequence of the second resolves an open problem of Nikiforov (see history notes in \cite{ALNS2025}).

\begin{theorem}[Ai, Lei, Ning, and Shi \cite{ALNS2025}]\label{thm:ALNS-cycle}
Let $n\ge k\ge5$, $t=\lfloor(k-1)/2\rfloor$, and $\mathcal P$ be a graph parameter feasible on connected graphs. If $G$ maximizes $\mathcal P$ among all 2-connected graphs of order $n$ containing no cycle of length at least $k$, then
$G\in\left\{K_s\vee\bigl((n-k+s)K_1\cup K_{k-2s}\bigr):2\le s\le t\right\}.$
\end{theorem}

\begin{theorem}[Ai, Lei, Ning, and Shi \cite{ALNS2025}]\label{thm:ALNS-path}
Let $n\ge k\ge4$, $t=\lfloor k/2\rfloor-1$, and $\mathcal P$ be a graph parameter feasible on connected graphs. If $G$ maximizes $\mathcal P$ among all connected graphs of order $n$ containing no $P_k$, then
$G\in\left\{K_s\vee\bigl((n-k+s+1)K_1\cup K_{k-2s-1}\bigr):1\le s\le t\right\}.$
\end{theorem}

\begin{figure}[t]
\centering
\begin{tikzpicture}[
  box/.style={draw,rounded corners,minimum width=24mm,minimum height=22mm,align=center},
  joinline/.style={thick},
  every node/.style={font=\small}]
\node[box] (A) at (0,0) {$K_s$\\[-1mm]$\bullet\ \bullet\ \bullet$};
\node[box] (B) at (4.1,1.35) {$sK_1$\\[-1mm]$\bullet\quad\bullet\quad\bullet$};
\node[box] (C) at (4.1,-1.35) {$K_{n-2s}$\\[-1mm]$\bullet\!\!\!\rule{7mm}{0.3pt}\!\!\!\bullet$};
\draw[joinline] (A.east) -- (B.west) node[midway,above,sloped]{all edges};
\draw[joinline] (A.east) -- (C.west) node[midway,below,sloped]{all edges};
\node[draw=none,fit=(B)(C),inner sep=3mm,label=right:$sK_1\cup K_{n-2s}$] {};
\end{tikzpicture}
\caption{The graph $K_s\vee(sK_1\cup K_{n-2s})$. No edges join $sK_1$ to $K_{n-2s}$.}
\label{fig:canonical-cycle}
\end{figure}

All graphs considered in this paper are finite and simple. Graph equality is defined up to isomorphism, and every graph parameter is assumed to be a real-valued isomorphism invariant. We follow the notation of Bondy and Murty \cite{BM2008}. For $S\subseteq V(G)$, let $G[S]$ be the subgraph induced by $S$, let $G-S=G[V(G)\setminus S]$, and let
$N_G(S)=\{v\in V(G)\setminus S:uv\in E(G)\text{ for some }u\in S\}$.
For disjoint sets $S,T\subseteq V(G)$, let $e_G(S,T)$ be the number of edges with one end in $S$ and the other in $T$. We write $C_n$ and $P_n$ for the cycle and the path of order $n$, respectively, and $K_0$ for the empty graph. A graph $G$ is \emph{Hamiltonian-connected} if every pair of distinct vertices of $G$ are joined by a Hamilton path; we write $\HC$ for this property.

For a graph $F$, a graph parameter $\mathcal P$, and an integer $k\ge0$, let $\ex_{\mathcal P}(n,F;\delta\ge k)$ denote the maximum value of $\mathcal P$ over all graphs of order $n$ and minimum degree at least $k$ that contain no copy of $F$. The corresponding extremal family is denoted by $\EX_{\mathcal P}(n,F;\delta\ge k)$. We define $\ex_{\mathcal P}(n,\HC;\delta\ge k)$ and $\EX_{\mathcal P}(n,\HC;\delta\ge k)$ analogously for non-Hamiltonian-connected graphs.

For the relevant integers $n$ and $s$, put
\begin{equation}\label{eq:canonical-graphs}\tag{1.2}
\begin{aligned}
H_{n,s}&=K_s\vee(sK_1\cup K_{n-2s}),\\
L_{n,s}&=K_{s-1}\vee(sK_1\cup K_{n-2s+1}),\\
R_{n,s}&=K_{s+1}\vee(sK_1\cup K_{n-2s-1}).
\end{aligned}
\end{equation}

We shall prove the following result, which generalizes \cref{thm:erdos1962} and is motivated by \cref{thm:ALNS-cycle,thm:ALNS-path}.

\begin{theorem}\label{ex-P-CP}
The following statements hold.
\begin{enumerate}[label=(\roman*)]
\item Let $\mathcal P$ be feasible, $k\ge1$, and $n\ge2k+1$. Then
\[
\ex_{\mathcal P}(n,C_n;\delta\ge k)
=\max_{k\le s\le\lfloor(n-1)/2\rfloor}\mathcal P(H_{n,s}),
\]
and
\[
\EX_{\mathcal P}(n,C_n;\delta\ge k)
=\left\{H_{n,s}:
\substack{k\le s\le\lfloor(n-1)/2\rfloor,
\mathcal P(H_{n,s})=\ex_{\mathcal P}(n,C_n;\delta\ge k)}
\right\}.
\]
\item Let $\mathcal P$ be feasible, $k\ge1$, and $n\ge2k+2$. Then
\[
\ex_{\mathcal P}(n,P_n;\delta\ge k)
=\max_{k+1\le s\le\lfloor n/2\rfloor}\mathcal P(L_{n,s}),
\]
and
\[
\EX_{\mathcal P}(n,P_n;\delta\ge k)
=\left\{L_{n,s}:
\substack{k+1\le s\le\lfloor n/2\rfloor,
\mathcal P(L_{n,s})=\ex_{\mathcal P}(n,P_n;\delta\ge k)}
\right\}.
\]
\end{enumerate}
\end{theorem}

\begin{remark}
Taking $\mathcal P=e$ (the number of edges) in \cref{ex-P-CP}(i), together with the convexity calculation in \cref{sec:applications}, recovers Erd\H{o}s's 1962 extremal theorem. Taking $\mathcal P=\rho$ or $\mathcal P=q$, where $\rho(G)$ and $q(G)$ are the spectral radius and the signless Laplacian spectral radius of $G$, respectively, gives the spectral formulas stated in \cref{sec:applications}. Thus, \cref{ex-P-CP} provides a common generalization for the edge and spectral versions.
\end{remark}

We shall also prove the following extremal theorem for Hamiltonian-connectedness. For the edge parameter, Ore \cite{Ore1963} determined the unrestricted maximum size of a non-Hamilton-connected graph. Ho, Lin, Tan, Hsu, and Hsu \cite{Ho2010} later obtained an Erd\H{o}s-type threshold under a prescribed minimum-degree lower bound, and Zhang \cite{Zhang2026} determined the maximum size and the extremal graphs for prescribed order and minimum degree.

\begin{theorem}\label{ex-P-HC}
Let $\mathcal P$ be feasible, $k\ge2$, and $n\ge2k$. Then
\[
\ex_{\mathcal P}(n,\HC;\delta\ge k)
=\max_{k-1\le s\le\lfloor n/2\rfloor-1}\mathcal P(R_{n,s}),
\]
and
\[
\EX_{\mathcal P}(n,\HC;\delta\ge k)
=\left\{R_{n,s}:
\substack{k-1\le s\le\lfloor n/2\rfloor-1,
\mathcal P(R_{n,s})=\ex_{\mathcal P}(n,\HC;\delta\ge k)}
\right\}.
\]
\end{theorem}

In \cref{sec:applications}, we present a theorem of F\"uredi, Kostochka, and Luo that follows from our main results. We also give applications to spectral graph theory, including the cycle, path, and Hamiltonian-connected extremal formulas. In \cref{sec:proofs}, we first state the low-degree and closure results used in the argument, then establish the propositions for two Kelmans operations, and finally prove \cref{ex-P-CP,ex-P-HC} by a sequence of claims.

%-----------------------------------------------------------------------
\section{Applications}\label{sec:applications}

We first record the concept of weak-feasibility introduced in \cite{ALNS2025} for connected graphs.

\begin{definition}[Ai, Lei, Ning, and Shi \cite{ALNS2025}]\label{def:weak-connected}
A graph parameter $\mathcal P$ is \emph{weakly feasible on connected graphs} if, for every connected graph $G$, we have $\mathcal P(G+xy)\ge\mathcal P(G)$ for every non-edge $xy$ of $G$, and $\mathcal P(G[x\to y])\ge\mathcal P(G)$ for every ordered pair of distinct vertices $x,y\in V(G)$.
\end{definition}

For later use, we also extend this definition to arbitrary graphs.

\begin{definition}\label{def:weak-feasible}
A graph parameter $\mathcal P$ is \emph{weakly feasible} if the following conditions hold for every graph $G$.
\begin{enumerate}[label=(W\arabic*)]
\item $\mathcal P(G+F)\ge\mathcal P(G)$ for every set
$F\subseteq\binom{V(G)}2\setminus E(G)$ such that $G+F$ is connected;
\item $\mathcal P(G[x\to y])\ge\mathcal P(G)$ for every ordered pair of distinct vertices $x,y\in V(G)$.
\end{enumerate}
\end{definition}

The next lemma are basic for our results.

\begin{lemma}\label{lem:basic-feasible}
The parameters $e$, $\rho$, and $q$ are feasible in the sense of \cref{def:feasible}.
\end{lemma}

\begin{proof}
The assertion for $e$ follows from $e(G+F)=e(G)+|F|$ and the fact that a Kelmans operation preserves the number of edges.

We next consider $\rho$. Let $F$ be a nonempty set of nonedges such that $G+F$ is connected. We have $A(G)\le A(G+F)$ entrywise and $A(G)\ne A(G+F)$, while $A(G+F)$ is irreducible. The comparison form of the Perron--Frobenius theorem (see, for example, \cite[Chapter~2]{Zhan2013}) gives $\rho(G+F)>\rho(G)$.

For the Kelmans operation parts, the proof is standard. We include the proof here for completeness. Let $H=G[x\to y]$ and put $W=N_G(x)\setminus N_G[y]$. The graphs $G[x\to y]$ and $G[y\to x]$ are isomorphic via the transposition of $x$ and $y$. Let $\bm z$ be a nonnegative unit eigenvector of $A(G)$ corresponding to $\rho(G)$. If $z_y\ge z_x$, then
\[
\bm z^{\mathsf T}A(H)\bm z-\bm z^{\mathsf T}A(G)\bm z
=2(z_y-z_x)\sum_{w\in W}z_w\ge0.
\tag{2.1}\label{eq:rho-kelmans}
\]
Hence $\rho(H)\ge\rho(G)$. If $z_y<z_x$, the same calculation applied to $G[y\to x]$, together with the isomorphism above, gives the same conclusion.

The proof for $q$ is analogous. Let $\bm z$ be a nonnegative unit eigenvector of $Q(G)$ corresponding to $q(G)$. Since
$\bm z^{\mathsf T}Q(G)\bm z=\sum_{uv\in E(G)}(z_u+z_v)^2$, if $z_y\ge z_x$, then
\[
\bm z^{\mathsf T}Q(H)\bm z-\bm z^{\mathsf T}Q(G)\bm z
=(z_y-z_x)\sum_{w\in W}(z_x+z_y+2z_w)\ge0.
\tag{2.2}\label{eq:q-kelmans}
\]
The opposite inequality between $z_x$ and $z_y$ is handled by interchanging $x$ and $y$. Moreover, if $G+F$ is connected, then $Q(G)\le Q(G+F)$ entrywise, $Q(G)\ne Q(G+F)$, and $Q(G+F)$ is irreducible. Thus, $q(G+F)>q(G)$ by the Perron--Frobenius theorem. We infer that $\rho$ and $q$ satisfy both conditions in \cref{def:feasible}.
\end{proof}

For an integer $r\ge2$, let $N_r(G)$ denote the number of copies of $K_r$ in $G$.

\begin{lemma}\label{lem:clique-weak}
For every $r\ge2$, the parameter $N_r$ is weakly feasible on arbitrary graphs.
\end{lemma}

\begin{proof}
Adding edges cannot decrease $N_r(G)$. For the Kelmans operation, we use the injection from \cite{ALNS2025}. Let $H=G[x\to y]$ and $W=N_G(x)\setminus N_G[y]$. Every $r$-clique destroyed by the operation contains $x$ and at least one vertex of $W$. Replace $x$ by $y$ in the clique. Each remaining vertex of the clique is either a common neighbor of $x$ and $y$, or belongs to $W$ and becomes adjacent to $y$ in $H$. We therefore obtain an $r$-clique of $H$ that was not an $r$-clique of $G$. This correspondence is injective, and hence $N_r(H)\ge N_r(G)$.
\end{proof}

The strictness in Condition (P1) of \cref{def:feasible} is used only to classify equality graphs. Without strictness, the extremal values remain unchanged.

\begin{theorem}\label{thm:weak-extremal}
The following statements hold.
\begin{enumerate}[label=(\roman*)]
\item If $\mathcal P$ is weakly feasible, $d\ge1$, and $n\ge2d+1$, then
\[
\ex_{\mathcal P}(n,C_n;\delta\ge d)
=\max_{d\le s\le\lfloor(n-1)/2\rfloor}\mathcal P(H_{n,s}).
\]
\item If $\mathcal P$ is weakly feasible, $d\ge1$, and $n\ge2d+2$, then
\[
\ex_{\mathcal P}(n,P_n;\delta\ge d)
=\max_{d+1\le s\le\lfloor n/2\rfloor}\mathcal P(L_{n,s}).
\]
\end{enumerate}
\end{theorem}

\begin{proof}
We first prove (i).

\begin{claim}\label{clm:weak-cycle-upper}
Every non-Hamiltonian graph $G$ of order $n$ with $\delta(G)\ge d$ satisfies
\[
\mathcal P(G)\le\max_{d\le s\le\lfloor(n-1)/2\rfloor}\mathcal P(H_{n,s}).
\]
\end{claim}

By \cref{thm:posa}, there is an integer $s$ with
$1\le s\le\lfloor(n-1)/2\rfloor$ for which $G$ contains a set $S=\{u_1,\dots,u_s\}$ such that $d_G(u_i)\le s$ for every $i$. Since $\delta(G)\ge d$, we have $s\ge d$. Put $T=V(G)\setminus S$. For $2\le j\le s$, we have
$|N_G(u_j)\cap(T\cup\{u_1,\dots,u_{j-1}\})|\le d_G(u_j)\le s\le n-s$.
Thus, \cref{lem:compression} applies with $b=s$. Since its construction uses only Kelmans operations and edge additions, conditions (W1) and (W2) give
$\mathcal P(G)\le\mathcal P(H_{n,s})$, which proves the claim.

\begin{claim}\label{clm:weak-cycle-lower}
Every graph $H_{n,s}$ with $d\le s\le\lfloor(n-1)/2\rfloor$ is admissible for part (i).
\end{claim}

The graph $H_{n,s}$ has minimum degree $s\ge d$. Deleting its central $K_s$ leaves $s+1$ components, whereas deleting $s$ vertices from a Hamilton cycle leaves at most $s$ path components. Thus, $H_{n,s}$ is non-Hamiltonian. Together with \cref{clm:weak-cycle-upper}, this proves (i).

We next prove (ii).

\begin{claim}\label{clm:weak-path-upper}
Every non-traceable graph $G$ of order $n$ with $\delta(G)\ge d$ satisfies
\[
\mathcal P(G)\le\max_{d+1\le s\le\lfloor n/2\rfloor}\mathcal P(L_{n,s}).
\]
\end{claim}

By \cref{lem:path-low-degree}, there is an integer $s$ with
$1\le s\le\lfloor n/2\rfloor$ for which $G$ contains a set $S=\{u_1,\dots,u_s\}$ such that $d_G(u_i)\le s-1$ for every $i$. Since $\delta(G)\ge d$, we have $s\ge d+1$. Put $T=V(G)\setminus S$. For $2\le j\le s$, we have
$|N_G(u_j)\cap(T\cup\{u_1,\dots,u_{j-1}\})|\le d_G(u_j)\le s-1\le n-s$.
Thus, \cref{lem:compression} applies with $b=s-1$, and conditions (W1) and (W2) give
$\mathcal P(G)\le\mathcal P(L_{n,s})$, which proves the claim.

\begin{claim}\label{clm:weak-path-lower}
Every graph $L_{n,s}$ with $d+1\le s\le\lfloor n/2\rfloor$ is admissible for part (ii).
\end{claim}

The graph $L_{n,s}$ has minimum degree $s-1\ge d$. Deleting its central $K_{s-1}$ leaves $s+1$ components, whereas deleting $s-1$ vertices from a spanning path leaves at most $s$ path components. We infer that $L_{n,s}$ is non-traceable. Together with \cref{clm:weak-path-upper}, this proves (ii).
\end{proof}

For fixed $r\ge2$, a direct count gives
\[
g_{n,r}(s):=N_r(H_{n,s})
=s\binom{s}{r-1}+\binom{n-s}{r}.
\tag{2.3}\label{eq:gr}
\]
Indeed, an $r$-clique either lies in the clique of order $n-s$, or contains one vertex of the independent set and $r-1$ vertices of the central $K_s$. With the convention $\binom ab=0$ for $b<0$ or $b>a$, we have
\[
g_{n,r}(s+2)-2g_{n,r}(s+1)+g_{n,r}(s)
=r\binom{s}{r-2}+(r-1)\binom{s}{r-3}
 +\binom{n-s-2}{r-2}\ge0.
\tag{2.4}\label{eq:convexity}
\]
Thus, $g_{n,r}$ is discretely convex, so its maximum on an integer interval is attained at an endpoint. We therefore recover the following theorem.

\begin{theorem}[F\"uredi, Kostochka, and Luo \cite{FKL2018}]\label{thm:FKL-cliques}
Let $n,d,r$ be integers with $1\le d\le\lfloor(n-1)/2\rfloor$ and $r\ge2$. If $G$ is a non-Hamiltonian graph of order $n$ with $\delta(G)\ge d$, then
\[
N_r(G)\le
\max\left\{
N_r(H_{n,d}),
N_r(H_{n,t})
\right\},
\]
where $t=\lfloor(n-1)/2\rfloor$.
\end{theorem}

We next give spectral applications. Let $A(G)$ and $D(G)$ be the adjacency matrix and the diagonal degree matrix of $G$, respectively. The \emph{spectral radius} $\rho(G)$ is the largest eigenvalue of $A(G)$. The \emph{signless Laplacian spectral radius} $q(G)$ is the largest eigenvalue of $Q(G)=D(G)+A(G)$. For $\lambda\in \{\rho,q\}$, we use $\spex_{\lambda}$ for the corresponding extremal values, and $\SPEX_{\lambda}$ for the corresponding extremal families.

Motivated by Erd\H{o}s's theorem, Li and Ning \cite[Problems~1 and~2]{LN2016} formulated the corresponding extremal questions for $\rho$ and $q$. In the present notation, the corresponding graph problems, restricted to ranges in which the admissible classes are nonempty, are as follows.

\begin{problem}[\rm{cf. Li and Ning \cite[Problems 1 and 2]{LN2016}}]\label{prob:spectral-general}
For $k\ge1$, determine $\spex_\rho(n,C_n;\delta\ge k)$ and $\spex_q(n,C_n;\delta\ge k)$ when $n\ge2k+1$, and determine $\spex_\rho(n,P_n;\delta\ge k)$ and $\spex_q(n,P_n;\delta\ge k)$ when $n\ge2k+2$.
\end{problem}

If $n\leq 2k$, then $\delta(G)\geq k\geq n/2$, so $G$ is Hamiltonian by Dirac’s theorem. So
the lower bound $n\ge2k+2$ is left for the class in the path problem to be nonempty. Indeed, if $n=2k+1$ and $\delta(G)\ge k$, then $G\vee K_1$ has order $2k+2$ and minimum degree at least $k+1$. Dirac's theorem gives a Hamilton cycle in $G\vee K_1$, and deleting the added vertex gives a Hamilton path in $G$.

For the adjacency spectral Hamilton cycle problem, Fiedler and Nikiforov \cite{FN2010} settled the case $k=1$ and characterized the extremal graph as $H_{n,1}=K_1\vee(K_1\cup K_{n-2})$. Ning and Ge \cite{NG2015} established the case $k=2$ for $n\ge14$. Chen, Hou, and Qian \cite{CHQ2018} subsequently improved the order range and obtained corresponding signless Laplacian conditions. For general $k$, Li and Ning \cite{LN2016} obtained the following two results from a stability form of Erd\H{o}s's theorem.

\begin{theorem}[{Li and Ning \cite[Theorem~1.5]{LN2016}}]\label{thm:LN-rho}
The following statements hold.
\begin{enumerate}[label=(\roman*)]
\item If $k\ge1$ and
$n\ge\max\{6k+5,(k^2+6k+4)/2\}$, then
\[
\SPEX_\rho(n,C_n;\delta\ge k)=\{H_{n,k}\}.
\]
\item If $k\ge0$ and
$n\ge\max\{6k+10,(k^2+7k+8)/2\}$, then
\[
\SPEX_\rho(n,P_n;\delta\ge k)=\{L_{n,k+1}\}.
\]
\end{enumerate}
\end{theorem}

\begin{theorem}[{Li and Ning \cite[Theorem~1.8]{LN2016}}]\label{thm:LN-q}
The following statements hold.
\begin{enumerate}[label=(\roman*)]
\item If $k\ge1$ and
$n\ge\max\{6k+5,(3k^2+5k+4)/2\}$, then
\[
\SPEX_q(n,C_n;\delta\ge k)=\{H_{n,k}\}.
\]
\item If $k\ge0$ and
$n\ge\max\{6k+10,(3k^2+9k+8)/2\}$, then
\[
\SPEX_q(n,P_n;\delta\ge k)=\{L_{n,k+1}\}.
\]
\end{enumerate}
\end{theorem}

\cref{thm:LN-rho,thm:LN-q} solve the corresponding problems under order assumptions that are quadratic in $k$. Nikiforov \cite[Theorem~1.4]{N2016} proved a related sufficient condition in a different, cubic order regime: if $k\ge2$, $n\ge k^3+k+4$, $\delta(G)\ge k$, and $\rho(G)\ge n-k-1$, then $G$ is Hamiltonian unless
$G=K_1\vee(K_{n-k-1}\cup K_k)$ or $G=H_{n,k}$. Further spectral conditions for Hamiltonicity can be found in \cite{LLD2019,LN2016,LNP2018,LLP2018,ZBWL2021}.

Motivated by the earlier spectral extremal and sufficient-condition results in \cite{WYL2019,ZBWL2021-2,XZW2022}, we formulate the Hamilton-connected counterpart.

\begin{problem}\label{prob:spectral-HC}
Let $k\geq 2$ and $n\geq 2k$. Among all non-Hamilton-connected graphs of order $n$ and minimum degree at least $k$, determine the maximum spectral radius and the maximum signless Laplacian spectral radius.
\end{problem}

By \cref{lem:basic-feasible,ex-P-CP}, the quantities in \cref{prob:spectral-general} are given by the following formulas.

\begin{theorem}\label{thm:spectral-cycle-path}
Let $k\ge1$ and let $\lambda\in\{\rho,q\}$.
\begin{enumerate}[label=(\roman*)]
\item If $n\ge2k+1$, then
\[
\spex_\lambda(n,C_n;\delta\ge k)
=\max_{k\le s\le\lfloor(n-1)/2\rfloor}\lambda(H_{n,s}),
\]
and
\[
\SPEX_\lambda(n,C_n;\delta\ge k)
=\left\{H_{n,s}:
\substack{k\le s\le\lfloor(n-1)/2\rfloor,
\lambda(H_{n,s})=\spex_\lambda(n,C_n;\delta\ge k)}
\right\}.
\]
\item If $n\ge2k+2$, then
\[
\spex_\lambda(n,P_n;\delta\ge k)
=\max_{k+1\le s\le\lfloor n/2\rfloor}\lambda(L_{n,s}),
\]
and
\[
\SPEX_\lambda(n,P_n;\delta\ge k)
=\left\{L_{n,s}:
\substack{k+1\le s\le\lfloor n/2\rfloor,
\lambda(L_{n,s})=\spex_\lambda(n,P_n;\delta\ge k)}
\right\}.
\]
\end{enumerate}
\end{theorem}

Similarly, \cref{lem:basic-feasible,ex-P-HC} solve \cref{prob:spectral-HC}.

\begin{theorem}\label{thm:spectral-HC}
Let $k\ge2$, $n\ge2k$, and $\lambda\in\{\rho,q\}$. Then
\[
\spex_\lambda(n,\HC;\delta\ge k)
=\max_{k-1\le s\le\lfloor n/2\rfloor-1}\lambda(R_{n,s}),
\]
and
\[
\SPEX_\lambda(n,\HC;\delta\ge k)
=\left\{R_{n,s}:
\substack{k-1\le s\le\lfloor n/2\rfloor-1,
\lambda(R_{n,s})=\spex_\lambda(n,\HC;\delta\ge k)}
\right\}.
\]
\end{theorem}

The cited spectral results are stated under order assumptions that are quadratic or cubic in the minimum-degree parameter. The present proofs instead use the graph-operation method of \cite{ALNS2025}. The two procedures in \cref{alg:independent,alg:concentrate} preserve the low-degree structure needed in the argument, although a general Kelmans operation may change the minimum degree. A related use of graph compression appears in the work of Ma and Ning \cite{MN2020}.

Nikiforov \cite{Nikiforov2017} introduced
$A_\alpha(G)=\alpha D(G)+(1-\alpha)A(G)$ for $0\le\alpha\le1$. Let $\rho_\alpha(G)$ denote the largest eigenvalue of $A_\alpha(G)$. We use $\spex_{\rho_\alpha}$ and $\SPEX_{\rho_\alpha}$ for the corresponding extremal value and extremal family.

\begin{lemma}\label{lem:Aalpha-feasible}
For every $\alpha$ with $0\le\alpha<1$, the parameter $\rho_\alpha$ is feasible.
\end{lemma}

\begin{proof}
Let $F$ be a nonempty set of non-edges such that $G+F$ is connected. We have
$A_\alpha(G)\le A_\alpha(G+F)$ entrywise and $A_\alpha(G)\ne A_\alpha(G+F)$, while $A_\alpha(G+F)$ is irreducible because $\alpha<1$. The Perron--Frobenius theorem gives
$\rho_\alpha(G+F)>\rho_\alpha(G)$.

Let $H=G[x\to y]$, put $W=N_G(x)\setminus N_G[y]$, and let $\bm z$ be a nonnegative unit eigenvector of $A_\alpha(G)$ corresponding to $\rho_\alpha(G)$. If $z_y\ge z_x$, then
\begin{align}
&\bm z^{\mathsf T}A_\alpha(H)\bm z-
 \bm z^{\mathsf T}A_\alpha(G)\bm z \notag\\
&\qquad=(z_y-z_x)\sum_{w\in W}
\bigl(\alpha(z_x+z_y)+2(1-\alpha)z_w\bigr)\ge0.
\tag{2.5}\label{eq:Aalpha-kelmans}
\end{align}
Thus, $\rho_\alpha(H)\ge\rho_\alpha(G)$. If $z_y<z_x$, we apply the same calculation to $G[y\to x]$ and use the isomorphism between $G[x\to y]$ and $G[y\to x]$. Therefore, $\rho_\alpha$ satisfies both conditions in \cref{def:feasible}.
\end{proof}

\begin{corollary}\label{cor:Aalpha}
Let $0\le\alpha<1$. The conclusions of \cref{thm:spectral-cycle-path,thm:spectral-HC} remain valid with $\lambda$ replaced by $\rho_\alpha$.
\end{corollary}

\begin{proof}
This follows from \cref{lem:Aalpha-feasible,ex-P-CP,ex-P-HC}.
\end{proof}

The restriction $\alpha<1$ is necessary: when $\alpha=1$, the parameter is the maximum degree and need not increase strictly after the addition of an edge.

%-----------------------------------------------------------------------
\section{Proofs of the main theorems}\label{sec:proofs}

We first state the closure results used below. For an integer $t$, the \emph{$t$-closure} $\cl_t(G)$ is obtained by repeatedly joining two nonadjacent vertices whose current degree sum is at least $t$, until no such pair remains.

\begin{theorem}[Bondy--Chv\'{a}tal \cite{BC1976}; Wong \cite{Wong1997}]\label{thm:closure}
Let $G$ be a graph of order $n$.
\begin{enumerate}[label=(\roman*)]
\item The graph $G$ is Hamiltonian if and only if $\cl_n(G)$ is Hamiltonian.
\item The graph $G$ is Hamilton-connected if and only if $\cl_{n+1}(G)$ is Hamilton-connected.
\end{enumerate}
\end{theorem}

We also use the following low-degree form of P\'{o}sa's theorem.

\begin{theorem}[P\'{o}sa \cite{P1962}]\label{thm:posa}
If a graph $G$ of order $n\ge3$ is non-Hamiltonian, then there is an integer $s$ with $1\le s\le\lfloor(n-1)/2\rfloor$ for which $G$ contains a set $S$ of size $s$ such that $d_G(v)\le s$ for every $v\in S$.
\end{theorem}

The next elementary relation between Hamilton paths and Hamilton cycles is standard; see, for example, \cite[Chapter~10]{BM2008}.

\begin{lemma}\label{lem:join-path-cycle}
Let $G$ be a graph wit order tat least two. Then $G$ has a Hamilton path if and only if $G\vee K_1$ has a Hamilton cycle.
\end{lemma}

\begin{lemma}\label{lem:path-low-degree}
If a graph $G$ of order $n\ge3$ has no Hamilton path, then there is an integer $s$ with $1\le s\le\lfloor n/2\rfloor$ for which $G$ contains a set $S$ of size $s$ such that $d_G(v)\le s-1$ for every $v\in S$.
\end{lemma}

\begin{proof}
Put $G^+=G\vee K_1$, and let $w$ be the added vertex. By \cref{lem:join-path-cycle}, the graph $G^+$ is non-Hamiltonian. Applying \cref{thm:posa} to $G^+$, we obtain an integer $s$ with $1\le s\le\lfloor n/2\rfloor$ and a set $S$ of $s$ vertices having degree at most $s$ in $G^+$. Since $d_{G^+}(w)=n>s$, we have $w\notin S$. Therefore, $S\subseteq V(G)$ and $d_G(v)=d_{G^+}(v)-1\le s-1$ for every $v\in S$.
\end{proof}

For Hamilton-connectedness, we use the following result without repeating its proof.

\begin{lemma}[{Ma and Ning \cite[Lemma~2.10]{MN2020}}]\label{lem:HC-low-degree}
Let $G$ be a non-Hamilton-connected graph of order $n$ with $\delta(G)\ge2$. Then there exists an integer $t$ with $2\le t\le\lfloor n/2\rfloor$ such that $G$ contains a set of $t-1$ vertices, each of degree at most $t$.
\end{lemma}

We now give the two Kelmans-operation procedures. The first removes all edges inside a prescribed low-degree set. Its input condition is stated only in terms of vertices preceding $u_j$; the vertex $u_j$ itself is not included in the counting set.

\begin{algorithm}[H]
\caption{Making $S$ independent by Kelmans operations}\label{alg:independent}
\begin{algorithmic}[1]
\Require A graph $G$ with $V(G)=S\cup T$, where
$S=\{u_1,\dots,u_s\}$ and $T=\{v_1,\dots,v_{n-s}\}$, such that
$|N_G(u_j)\cap(T\cup\{u_1,\dots,u_{j-1}\})|\le n-s$ for $2\le j\le s$.
\Ensure A graph $\Gamma$ obtained from $G$ by Kelmans operations such that $S$ is independent.
\State $\Gamma\gets G$
\For{$i=1$ to $s-1$}
  \For{$j=i+1$ to $s$}
    \If{$u_iu_j\in E(\Gamma)$}
      \For{$t=1$ to $n-s$}
        \If{$u_jv_t\notin E(\Gamma)$}
          \State $\Gamma\gets\Gamma[u_i\to v_t]$
          \State \textbf{break}
        \EndIf
      \EndFor
    \EndIf
  \EndFor
\EndFor
\State \Return $\Gamma$
\end{algorithmic}
\end{algorithm}

\begin{lemma}\label{lem:algorithm-one}
\cref{alg:independent} is well defined and terminates. Its output $\Gamma$ has the following properties.
\begin{enumerate}[label=(\roman*)]
\item The set $S$ is independent in $\Gamma$.
\item For every $u_i\in S$, we have $d_\Gamma(u_i)\le d_G(u_i)$.
\item If $u_i$ is the source of at least one Kelmans operation, then $d_\Gamma(u_i)<d_G(u_i)$. Consequently, $d_\Gamma(u_i)=d_G(u_i)$ if and only if $u_i$ is never used as a source.
\end{enumerate}
\end{lemma}

\begin{proof}
\begin{claim}\label{clm:algorithm-one-target}
Whenever the pair $(i,j)$ is examined and $u_iu_j\in E(\Gamma)$, there exists $v_t\in T$ such that $u_jv_t\notin E(\Gamma)$.
\end{claim}

Fix $j$. Until the pair $(j-1,j)$ has been examined, every operation already performed has source $u_h$ with $h<j$ and target in $T$. Put
$c_j(H)=|N_H(u_j)\cap(T\cup\{u_1,\dots,u_{j-1}\})|$.
If such an operation affects $u_j$, then $u_j$ loses the neighbor $u_h$ and gains the target vertex in $T$. Hence $c_j$ is unchanged. We infer that, when $(i,j)$ is examined,
$c_j(\Gamma)=c_j(G)\le n-s$. If $u_iu_j\in E(\Gamma)$ and every vertex of $T$ were adjacent to $u_j$, then $c_j(\Gamma)\ge|T|+1=n-s+1$, a contradiction. This proves the claim.

\begin{claim}\label{clm:algorithm-one-degree}
In each operation, the degree of the source vertex in $S$ strictly decreases, while the degree of every other vertex of $S$ is unchanged.
\end{claim}

Suppose that the operation is $\Gamma[u_i\to v_t]$, and let
$W=N_\Gamma(u_i)\setminus N_\Gamma[v_t]$. By the choice of $v_t$, we have $u_j\in W$, so $W\ne\varnothing$. The source $u_i$ loses all $|W|$ incident edges and gains no new incident edge. Thus, its degree strictly decreases. If $u_h\in S\setminus\{u_i\}$ belongs to $W$, then $u_h$ loses the edge $u_hu_i$ and gains the edge $u_hv_t$; otherwise its incident edges are unchanged. This proves the claim.

\begin{claim}\label{clm:algorithm-one-independent}
No operation creates an edge with both ends in $S$, and the edge $u_iu_j$ that triggers an operation is deleted.
\end{claim}

All new edges produced by the operation have the target $v_t\in T$ as one endpoint. Moreover, $u_j\in W$, so $u_iu_j$ is deleted. Thus, the claim follows.

The loops examine finitely many triples $(i,j,t)$, and the explicit break ends the search as soon as one operation has been performed for the current pair.
When the pair $(i,j)$ is processed, the edge $u_iu_j$
is either already absent or is deleted at that stage. By Claim \cref{clm:algorithm-one-target}, it is never recreated. Therefore, $S$ is independent in the output. Statements (ii) and (iii) follow from \cref{clm:algorithm-one-degree}.
\end{proof}

The second procedure concentrates all neighborhoods of $S$ in a fixed initial segment of $T$. It does not assert that $T$ becomes a clique.

\begin{algorithm}[H]
\caption{Concentrating the neighborhoods of $S$}\label{alg:concentrate}
\begin{algorithmic}[1]
\Require A graph $\Gamma$ with $V(\Gamma)=S\cup T$, where
$S=\{u_1,\dots,u_s\}$ is independent and
$T=\{v_1,\dots,v_{n-s}\}$. Put
$r=\max\{d_\Gamma(u_i):1\le i\le s\}$, relabel $S$ so that $d_\Gamma(u_1)=r$, and relabel $T$ so that
$N_\Gamma(u_1)=\{v_1,\dots,v_r\}$.
\Ensure A graph $\Gamma^*$ obtained from $\Gamma$ by Kelmans operations such that $S$ remains independent, every degree in $S$ is preserved, and
$N_{\Gamma^*}(u_i)\subseteq\{v_1,\dots,v_r\}$ for all $i$.
\State $\Gamma^*\gets\Gamma$
\For{$i=1$ to $s$}
  \For{$j=r+1$ to $n-s$}
    \If{$u_iv_j\in E(\Gamma^*)$}
      \For{$t=1$ to $r$}
        \If{$u_iv_t\notin E(\Gamma^*)$}
          \State $\Gamma^*\gets\Gamma^*[v_j\to v_t]$
          \State \textbf{break}
        \EndIf
      \EndFor
    \EndIf
  \EndFor
\EndFor
\State \Return $\Gamma^*$
\end{algorithmic}
\end{algorithm}

\begin{lemma}\label{lem:algorithm-two}
\cref{alg:concentrate} is well defined and terminates. Its output $\Gamma^*$ satisfies the following properties.
\begin{enumerate}[label=(\roman*)]
\item The set $S$ is independent in $\Gamma^*$.
\item We have $d_{\Gamma^*}(u_i)=d_\Gamma(u_i)$ for $1\le i\le s$.
\item We have $N_{\Gamma^*}(u_i)\subseteq\{v_1,\dots,v_r\}$ for $1\le i\le s$.
\end{enumerate}
\end{lemma}

\begin{proof}
\begin{claim}\label{clm:algorithm-two-target}
If $u_iv_j\in E(\Gamma^*)$ for some $j>r$, then $u_i$ has a nonneighbor in $\{v_1,\dots,v_r\}$.
\end{claim}

As shown in the next claim, all previous operations preserve the degree of $u_i$, and hence $d_{\Gamma^*}(u_i)=d_\Gamma(u_i)\le r$ at the relevant stage. If $u_i$ were adjacent to all of $v_1,\dots,v_r$ and also to $v_j$, then $d_{\Gamma^*}(u_i)\ge r+1$, a contradiction.

\begin{claim}\label{clm:algorithm-two-degree}
An operation $\Gamma^*[v_j\to v_t]$ preserves the degree of every vertex in $S$ and creates no edge inside $S$.
\end{claim}

Put $W=N_{\Gamma^*}(v_j)\setminus N_{\Gamma^*}[v_t]$. For $u_h\in S$, either $u_h\notin W$, in which case no incident edge of $u_h$ changes, or $u_h\in W$, in which case the edge $u_hv_j$ is replaced by $u_hv_t$. Thus, the degree of $u_h$ is preserved. Every new edge is incident with $v_t\in T$, so no edge inside $S$ is created.

\begin{claim}\label{clm:algorithm-two-progress}
The operation triggered by $u_iv_j$ deletes that edge and creates no new edge from $S$ to $\{v_{r+1},\dots,v_{n-s}\}$.
\end{claim}

The chosen vertex $v_t$ is a non-neighbor of $u_i$, so $u_i\in W$ and $u_iv_j$ is replaced by $u_iv_t$. Every new neighbor in $S$ created by the operation is joined to the target $v_t$ with $t\le r$. Every newly created edge incident with $S$ has $v_t$, where $t\leq r$, as its endpoint in $T$. This proves the claim.

By \cref{clm:algorithm-two-target}, the inner search succeeds whenever it is invoked, and the explicit break ensures that exactly one operation is performed for the current pair. The loops are finite. By \cref{clm:algorithm-two-degree}, $S$ remains independent and all degrees in $S$ are preserved. By \cref{clm:algorithm-two-progress}, every edge from $S$ to the terminal segment of $T$ is deleted when it is encountered and is never recreated. We obtain (i)--(iii).
\end{proof}

We combine the two procedures in the following lemma. For integers $n,s,b$ with $1\le b\le n-s$, let
\[
J_n(s,b)=K_b\vee\bigl(sK_1\cup K_{n-s-b}\bigr).
\tag{3.1}\label{eq:Jnsb}
\]

\begin{lemma}\label{lem:compression}
Let $G$ be a graph of order $n$ with a partition $V(G)=S\cup T$, where
$S=\{u_1,\dots,u_s\}$, $|T|=n-s$, and $d_G(u_i)\le b\le n-s$ for every $i$. Assume that the input condition of \cref{alg:independent} holds. Then there are graphs $\Gamma$ and $\Gamma^*$ such that
\begin{enumerate}[label=(\roman*)]
\item $\Gamma$ is obtained from $G$ by \cref{alg:independent};
\item $\Gamma^*$ is obtained from $\Gamma$ by \cref{alg:concentrate};
\item $\Gamma^*$ is a spanning subgraph of $J_n(s,b)$.
\end{enumerate}
Consequently, $\mathcal P(G)\le\mathcal P(J_n(s,b))$ for every feasible or weakly feasible parameter $\mathcal P$.
\end{lemma}

\begin{proof}
By \cref{lem:algorithm-one}, \cref{alg:independent} gives a graph $\Gamma$ in which $S$ is independent and $d_\Gamma(u_i)\le d_G(u_i)\le b$. Put
$r=\max\{d_\Gamma(u_i):1\le i\le s\}$, so $r\le b$, and relabel the vertices as required by \cref{alg:concentrate}. By \cref{lem:algorithm-two}, the resulting graph $\Gamma^*$ satisfies that $S$ is independent and
$N_{\Gamma^*}(u_i)\subseteq\{v_1,\dots,v_r\}\subseteq\{v_1,\dots,v_b\}$ for every $i$. The Kelmans operations give
\(\mathcal{P}(G) \leq \mathcal{P}(\Gamma^{*})\).
If \(\Gamma^{*}=J_n(s,b)\), there is nothing to prove.
Otherwise, let
$F=E\bigl(J_n(s,b)\bigr)\setminus E(\Gamma^{*}).$
Then \(F\neq\varnothing\), and
\(J_n(s,b)=\Gamma^{*}+F\) is connected since \(b\geq 1\).
Hence condition \textup{(P1)} gives
$\mathcal{P}(\Gamma^{*})<\mathcal{P}\bigl(J_n(s,b)\bigr).$
\end{proof}

The next lemma gives the equality argument that is common to all three main results.

\begin{lemma}\label{lem:equality-rigidity}
In the setting of \cref{lem:compression}, suppose that the output of Algorithm 2 is exactly the specified copy of $J_n(s,b)$. Then \cref{alg:independent} performs no operation, $G=\Gamma$, the set $S$ is independent in $G$, every vertex of $S$ has degree $b$ in $G$, and $G[T]$ is complete.
\end{lemma}

\begin{proof}
\begin{claim}\label{clm:rigidity-no-first}
\cref{alg:independent} performs no operation.
\end{claim}

Since $\Gamma^*=J_n(s,b)$, every vertex of $S$ has degree $b$ in $\Gamma^*$. \cref{alg:concentrate} preserves all degrees in $S$, and hence $d_\Gamma(u_i)=b$ for every $i$. By \cref{lem:algorithm-one},
$b=d_\Gamma(u_i)\le d_G(u_i)\le b$. Thus, $d_G(u_i)=d_\Gamma(u_i)=b$. If $u_i$ had been used as a source, then \cref{lem:algorithm-one}(iii) would give $d_\Gamma(u_i)<d_G(u_i)$, a contradiction. This proves the claim. Consequently, $G=\Gamma$ and $S$ is independent in $G$.

\begin{claim}\label{clm:rigidity-T}
The graph $G[T]$ is complete.
\end{claim}

Kelmans operations preserve the total number of edges, so $e(G)=e(\Gamma^*)$. Since $S$ is independent in $G$ and every vertex of $S$ has degree $b$, we have $e_G(S,T)=sb$. The same equality holds in $\Gamma^*=J_n(s,b)$. Therefore,
\[
e(G[T])=e(G)-sb=e(\Gamma^*)-sb=\binom{n-s}{2}.
\]
Thus, $G[T]$ is complete, proving the claim and the lemma.
\end{proof}

We shall also use the following unified closure calculation.

\begin{lemma}\label{lem:closure-completion}
Let $H$ be a graph of order $m$ with a partition $V(H)=S\cup T$. Suppose that $S$ is independent, $T$ is a clique, $|S|=a$, and $d_H(x)=b$ for every $x\in S$. Put $A=N_H(S)\subseteq T$. Let $\tau$ be an integer such that
$b+m-a\ge\tau$ and $2(m-a)\ge\tau$. If $|A|\ge b+1$, then $\cl_\tau(H)=K_m$.
\end{lemma}

\begin{proof}
\begin{claim}\label{clm:closure-A}
The $\tau$-closure adds every missing edge between $S$ and $A$.
\end{claim}

For $y\in A$, the clique structure of $T$ and the fact that $y$ has a neighbor in $S$ give $d_H(y)\ge m-a$. Hence, for every $x\in S$ nonadjacent to $y$, we have
$d_H(x)+d_H(y)\ge b+m-a\ge\tau$. Degrees only increase during the closure process, and the claim follows.

\begin{claim}\label{clm:closure-TminusA}
After the edges in \cref{clm:closure-A} have been added, the $\tau$-closure adds every edge between $S$ and $T\setminus A$.
\end{claim}

At this stage each vertex of $S$ has degree at least $|A|\ge b+1$. A vertex $z\in T\setminus A$ has no neighbor in $S$ in the original graph and has degree $m-a-1$ in the clique $T$. Therefore,
$(b+1)+(m-a-1)=b+m-a\ge\tau$, and the claim follows.

\begin{claim}\label{clm:closure-S}
The $\tau$-closure finally adds every missing edge inside $S$.
\end{claim}

After \cref{clm:closure-TminusA}, every vertex of $S$ is adjacent to all $m-a$ vertices of $T$. The degree sum of two vertices of $S$ is therefore at least $2(m-a)\ge\tau$. Thus, all missing edges inside $S$ are added. The resulting graph is $K_m$.
\end{proof}

We now prove the main theorems.

\begin{proof}[Proof of \cref{ex-P-CP}]
By Theorem 2.3, we have the extremal parts. 
For the part (i). 
let $M=\max\{\mathcal P(H_{n,s}):k\le s\le\lfloor(n-1)/2\rfloor\}$. By Theorem 2.3, we have the following claims.
 
\begin{claim}\label{clm:cycle-lower}
For every $s$ with $k\le s\le\lfloor(n-1)/2\rfloor$, the graph
$H_s:=H_{n,s}$ has minimum degree $s$ and is non-Hamiltonian.
\end{claim}

\begin{claim}\label{clm:cycle-upper}
Every non-Hamiltonian graph $G$ of order $n$ with $\delta(G)\ge k$ satisfies $\mathcal P(G)\le M$.
\end{claim}

\begin{claim}\label{clm:cycle-equality}
If $G\in\EX_{\mathcal P}(n,C_n;\delta\ge k)$, then
$G=H_s$ for some $s$ with $k\le s\le\lfloor(n-1)/2\rfloor$.
\end{claim}

Choose $S$ and $s$ as in the proof of \cref{clm:cycle-upper}, and let $\Gamma$ and $\Gamma^*$ be the outputs of the two algorithms. Since $\mathcal P(G)=M$, we have
\[
M=\mathcal P(G)\le\mathcal P(\Gamma^*)
\le\mathcal P(H_s)\le M.
\tag{3.2}\label{eq:cycle-equality-chain}
\]
Hence all inequalities in \eqref{eq:cycle-equality-chain} are equalities. The graph $H_s$ is connected and is obtained from $\Gamma^*$ by adding nonedges. If $\Gamma^*$ were a proper spanning subgraph of $H_s$, condition (P1) would give
$\mathcal P(H_s)>\mathcal P(\Gamma^*)$, contrary to \eqref{eq:cycle-equality-chain}. Therefore, $\Gamma^*=H_s$.

By \cref{lem:equality-rigidity} with $b=s$, \cref{alg:independent} performs no operation, $G=\Gamma$, the set $S$ is independent in $G$, every vertex of $S$ has degree $s$, and $T$ is a clique. Put $A=N_G(S)\subseteq T$. We have $|A|\ge s$. If $|A|=s$, then $N_G(x)=A$ for every $x\in S$, and hence $G=H_s$.

Suppose that $|A|\ge s+1$. Apply \cref{lem:closure-completion} to $G$ with
$m=n$, $a=b=s$, and $\tau=n$. We have
$b+m-a=n$ and $2(m-a)=2(n-s)\ge n$. Thus, $\cl_n(G)=K_n$. By \cref{thm:closure}(i), $G$ is Hamiltonian, a contradiction. We infer that $|A|=s$ and $G=H_s$. This proves the claim. Conversely, \cref{clm:cycle-lower} shows that every $H_{n,s}$ in the stated range whose $\mathcal P$-value equals the displayed maximum is extremal. Hence the asserted equality for the extremal family follows, and the proof of part (i) is complete.
\end{proof}

\begin{proof}[Proof of \cref{ex-P-CP}(ii)]
Put $M=\max\{\mathcal P(L_{n,s}):k+1\le s\le\lfloor n/2\rfloor\}$. By Theorem 2.3, we have the following.

\begin{claim}\label{clm:path-lower}
For every $s$ with $k+1\le s\le\lfloor n/2\rfloor$, the graph
$L_s:=L_{n,s}$ has minimum degree $s-1$ and has no Hamilton path.
\end{claim}

\begin{claim}\label{clm:path-upper}
Every non-traceable graph $G$ of order $n$ with $\delta(G)\ge k$ satisfies $\mathcal P(G)\le M$.
\end{claim}

\begin{claim}\label{clm:path-equality}
If $G\in\EX_{\mathcal P}(n,P_n;\delta\ge k)$, then
$G=L_s$ for some $s$ with $k+1\le s\le\lfloor n/2\rfloor$.
\end{claim}

Choose $S$ and $s$ as in \cref{clm:path-upper}, and let $\Gamma$ and $\Gamma^*$ be the outputs of the two algorithms. As in \eqref{eq:cycle-equality-chain}, extremality gives
\[
M=\mathcal P(G)\le\mathcal P(\Gamma^*)
\le\mathcal P(L_s)\le M.
\tag{3.3}\label{eq:path-equality-chain}
\]
Since $L_s$ is connected, the strictness in (P1) implies $\Gamma^*=L_s$. Applying \cref{lem:equality-rigidity} with $b=s-1$, we obtain $G=\Gamma$, $S$ is independent in $G$, every vertex of $S$ has degree $s-1$, and $T$ is a clique.

Put $A=N_G(S)\subseteq T$. If $|A|=s-1$, then every vertex of $S$ is adjacent to all vertices of $A$, and hence $G=L_s$. Suppose that $|A|\ge s$. Let $G^+=G\vee K_1$, let $w$ be the added vertex, and put $T^+=T\cup\{w\}$. Then $T^+$ is a clique, $S$ is independent, and every vertex of $S$ has degree $s$ in $G^+$. Moreover,
$N_{G^+}(S)=A\cup\{w\}$ has size at least $s+1$.

Apply \cref{lem:closure-completion} to $G^+$ with
$m=n+1$, $a=b=s$, and $\tau=n+1$. We have
$b+m-a=n+1$ and $2(m-a)=2(n+1-s)\ge n+1$. Thus,
$\cl_{n+1}(G^+)=K_{n+1}$. By \cref{thm:closure}(i), the graph $G^+$ is Hamiltonian. It follows from \cref{lem:join-path-cycle} that $G$ has a Hamilton path, a contradiction. Therefore, $|A|=s-1$ and $G=L_s$. This proves the claim. Conversely, \cref{clm:path-lower} shows that every $L_{n,s}$ in the stated range whose $\mathcal P$-value equals the displayed maximum is extremal. Hence the asserted equality for the extremal family follows, and the proof of part (ii) is complete.
\end{proof}

\begin{proof}[Proof of \cref{ex-P-HC}]
Put $M=\max\{\mathcal P(R_{n,s}):k-1\le s\le\lfloor n/2\rfloor-1\}$.

\begin{claim}\label{clm:HC-lower}
For every $s$ with $k-1\le s\le\lfloor n/2\rfloor-1$, the graph
$R_s:=R_{n,s}$ has minimum degree $s+1$ and is not Hamilton-connected.
\end{claim}

The vertices in the independent set have degree $s+1$, and the vertices in $K_{n-2s-1}$ have degree $n-s-1\ge s+1$. Thus, $\delta(R_s)=s+1\ge k$. Let $X$ be the central $K_{s+1}$ and choose distinct vertices $p,q\in X$. The graph $R_s-X$ has $s+1$ components. If a Hamilton $p$--$q$ path existed, then deleting the $s+1$ vertices of $X$ from that path would leave at most $|X|-1=s$ path segments, because both ends of the path lie in $X$. These segments cannot cover all $s+1$ components of $R_s-X$. We infer that $R_s$ is not Hamilton-connected, and hence
$\ex_{\mathcal P}(n,\HC;\delta\ge k)\ge M$.

\begin{claim}\label{clm:HC-upper}
Every non-Hamilton-connected graph $G$ of order $n$ with $\delta(G)\ge k$ satisfies $\mathcal P(G)\le M$.
\end{claim}

By \cref{lem:HC-low-degree}, there is an integer $t$ with
$2\le t\le\lfloor n/2\rfloor$ for which $G$ contains a set of $t-1$ vertices, each of degree at most $t$. Put $s=t-1$. Then
$1\le s\le\lfloor n/2\rfloor-1$, and there is a set
$S=\{u_1,\dots,u_s\}$ such that $d_G(u_i)\le s+1$ for all $i$. Since $\delta(G)\ge k$, we have $s\ge k-1$. Put $T=V(G)\setminus S$. The input condition of \cref{alg:independent} follows from
$s+1\le n-s$. Applying \cref{lem:compression} with $b=s+1$, we obtain
$\mathcal P(G)\le\mathcal P(J_n(s,s+1))=\mathcal P(R_s)\le M$. This proves the claim and the extremal-value equality.

\begin{claim}\label{clm:HC-equality}
If $G\in\EX_{\mathcal P}(n,\HC;\delta\ge k)$, then
$G=R_s$ for some $s$ with $k-1\le s\le\lfloor n/2\rfloor-1$.
\end{claim}

Choose $S$ and $s$ as in \cref{clm:HC-upper}, and let $\Gamma$ and $\Gamma^*$ be the outputs of the two algorithms. Extremality gives
\[
M=\mathcal P(G)\le\mathcal P(\Gamma^*)
\le\mathcal P(R_s)\le M.
\tag{3.4}\label{eq:HC-equality-chain}
\]
Since $R_s$ is connected, condition (P1) implies $\Gamma^*=R_s$. Applying \cref{lem:equality-rigidity} with $b=s+1$, we obtain $G=\Gamma$, $S$ is independent in $G$, every vertex of $S$ has degree $s+1$, and $T$ is a clique.

Put $A=N_G(S)\subseteq T$. If $|A|=s+1$, then every vertex of $S$ is adjacent to all vertices of $A$, and hence $G=R_s$. Suppose that $|A|\ge s+2$. Apply \cref{lem:closure-completion} to $G$ with
$m=n$, $a=s$, $b=s+1$, and $\tau=n+1$. We have
$b+m-a=n+1$ and, since $s\le\lfloor n/2\rfloor-1$,
$2(m-a)=2(n-s)\ge n+1$. Therefore,
$\cl_{n+1}(G)=K_n$. By \cref{thm:closure}(ii), the graph $G$ is Hamilton-connected, a contradiction. Thus, $|A|=s+1$ and $G=R_s$. This proves the claim. Conversely, \cref{clm:HC-lower} shows that every $R_{n,s}$ in the stated range whose $\mathcal P$-value equals the displayed maximum is extremal. Hence the asserted equality for the extremal family follows, and the proof is complete.
\end{proof}

\end{document}